\documentclass[12pt,a4paper]{amsart}
\usepackage{amsmath}
\usepackage{amsthm}
\usepackage{amssymb}
\usepackage{amscd}
\usepackage{drpack}
\usepackage[english]{babel}
\usepackage{verbatim}
\usepackage[shortlabels]{enumitem}
\usepackage{tikz-cd}

\newtheoremstyle{mio}%
{}{} 
{\itshape}{} 
{\bfseries}{.}{ } 
{#1 #2\thmnote{~\mdseries(#3)}} 
\theoremstyle{mio}
\newtheorem{teor}{Theorem}[section]
\newtheorem{cor}[teor]{Corollary}
\newtheorem{prop}[teor]{Proposition}
\newtheorem{lemma}[teor]{Lemma}
\newtheorem{defin}[teor]{Definition}

\newtheoremstyle{definition2}%
{}{} 
{}{} 
{\bfseries}{.}{ } 
{#1 #2\thmnote{\mdseries~ #3}} 
\theoremstyle{definition2}
\newtheorem{ex}[teor]{Example}
\newtheorem{oss}[teor]{Remark}

\title{Smooth sequences and overring operators}
\author{Dario Spirito}
\address{Dipartimento di Scienze Matematiche, Fisiche e Informatiche, Universit\`a di Udine, Udine, Italy}
\email{dario.spirito@uniud.it}
\keywords{Smooth sequences; free groups; invertible ideals; Pr\"ufer domains; overrings}
\subjclass[2020]{13A15; 13F05; 13G05; 20K20}

\newcommand{\class}{\mathcal{C}}
\newcommand{\deriv}{\mathcal{D}}
\newcommand{\Inv}{\mathrm{Inv}}
\newcommand{\V}{\mathcal{V}}

\newcommand{\inverse}{\mathrm{inv}}
\newcommand{\rk}{\mathrm{rk}}
\DeclareMathSymbol{\mh}{\mathord}{operators}{`\-}
\newcommand{\quot}{\mathcal{Q}}
\newcommand{\Crit}{\mathrm{Crit}}

\begin{document}
\begin{abstract}
We introduce smooth sequences of integral domains as well-ordered ascending chains that behave well at limit ordinals. Subsequently, we use this notion to give some conditions on the freeness of kernels of extension maps between groups of invertible ideals of Pr\"ufer domains. We also define overring operators to construct smooth sequences in a recursive way.
\end{abstract}
\maketitle

\section{Introduction}
Let $D$ be an integral domain with quotient field $K$. A \emph{fractional ideal} of $D$ is a $D$-submodule $I$ of $K$ such that $dI\subseteq D$ for some $d\in D\setminus\{0\}$. A fractional ideal is said to be \emph{invertible} if $IJ=D$ for some fractional ideal $J$; the set $\Inv(D)$ of the invertible ideals of $D$ is a group under the product of ideals.

It is in general difficult to determine the structure of $\Inv(D)$; two of the trivial cases are for Dedekind domains (where prime factorization shows that $\Inv(D)$ is free with the maximal space $\Max(D)$ as a basis) and for valuation domains (where $\Inv(D)$ is isomorphic to the value group). Recently, an approximation method has been introduced to analyze more cases, among which we mention almost Dedekind domain \cite{SP-scattered,bounded-almded}, strongly discrete Pr\"ufer domains \cite{InvXD} and rings endowed with a pre-Jaffard family \cite{inv-free}. This method involves setting up a well-ordered chain $\{D_\alpha\}$ of overrings of $D$ corresponding to a well-ordered chain of subgroups of $\Inv(D)$ (more precisely, the subgroup corresponding to $D_\alpha$ is the kernel of the extension homomorphism $\Inv(D)\longrightarrow\Inv(D_\alpha)$), and using a well-known result about smooth sequences of subgroups to prove the freeness of $\Inv(D)$ under suitable hypothesis.

In this paper, we abstract this method by introducing the concept of a \emph{smooth sequences} of rings (or, more generally, of sets) as an ascending well-ordered chain that behaves well at limit ordinals (see Definition \ref{defin:smooth}). We show that a smooth sequence $\{D_\alpha\}$ of overrings of a domain $D$ gives rise to a smooth sequence of subgroups of $\Inv(D)$ (Proposition \ref{prop:smooth-caratt-inv}); in the case of Pr\"ufer domains, we show that smooth sequence of overrings correspond to smooth sequences of open sets of the maximal space, with respect to the inverse topology (Proposition \ref{prop:smooth-caratt}), setting up a link with the construction of the derived sequence of a topological space. We prove three technical criteria that allow to reduce the problem of proving that $\Inv(D)$ is free to the freeness of $\Inv(T)$ for some overring $T$; the first one involves the kernels of the extension morphisms $\Inv(D_\alpha)\longrightarrow\Inv(D_{\alpha+1})$ (Theorem \ref{teor:rings-smooth}), while the other two are more powerful, but apply only in the context of one-dimensional Pr\"ufer domains (Theorems \ref{teor:kernel-free-dvr} and \ref{teor:unione-numerabile}). The proofs of these results are obtained as abstractions of the results used in the previous papers \cite{SP-scattered,bounded-almded,InvXD,inv-free}, and will be used to simplify a more in-depth analysis of the group of invertible ideals in the setting of one-dimensional Pr\"ufer domains \cite{grp-bounded-ratio}.

In Section \ref{sect:ovop} we introduce an easy way to construct smooth sequences by means of iterating an \emph{overring operator}, i.e., a function that associates to a domain $D$ an overring of $D$. We introduce the rank of a domain with respect to an overring operator $\Lambda$ and the rank of $\Lambda$ itself as extensions of the notion of the SP-rank of an almost Dedekind domain (which in turn is inspired by the definition of the Cantor-Bendixson rank of an integral domain; see \cite{SP-scattered}). We also give several examples of overring operators and of the corresponding smooth sequences.

\medskip

Throughout the paper, all rings are commutative and unitary; we shall also only deal with integral domains, i.e., rings without zerodivisors. If $D$ is an integral domain, we denote by $\quot(D)$ its quotient field; an \emph{overring} of $D$ is a ring between $D$ and $\quot(D)$.

If $D,T$ are domains with $D\subseteq T$ (in particular, if $T$ is an overring of $D$) then there is a natural \emph{extension map} 
\begin{equation*}
\begin{aligned}
\Inv(D)  & \longrightarrow \Inv(T),\\
I & \longmapsto IT,
\end{aligned}
\end{equation*}
which is a homomorphism of groups.

An integral domain $D$ is a \emph{valuation domain} if, for every $f\in\quot(D)\setminus\{0\}$, at least one of $f$ and $1/f$ is in $D$. A \emph{Pr\"ufer domain} is an integral domain $D$ such that $D_P$ is a valuation domain for every prime ideal $P$ of $D$; equivalently, $D$ is a Pr\"ufer domain if and only if every nonzero finitely generated (fractional) ideal is invertible. Every overring of a Pr\"ufer domain is a Pr\"ufer domain \cite[Theorem 26.1(a)]{gilmer}. A \emph{discrete valuation domain} (DVR) is a Noetherian valuation domain; an \emph{almost Dedekind domain} is a domain $D$ such that $D_M$ is a DVR for every maximal ideal $M$. An almost Dedekind domain is Pr\"ufer and, if it is not a field, it is one-dimensional.

The \emph{inverse topology} on the spectrum $\Spec(D)$ of $D$ is the topology whose basic open sets are in the form $\V(I):=\{P\in\Spec(D)\mid I\subseteq P\}$, as $I$ ranges among the finitely generated ideals of $D$. We denote by $\Spec(D)^\inverse$ this space, and by $\Max(D)^\inverse$ the maximal space endowed with the inverse topology. The space $\Spec(D)^\inverse$ is always compact; moreover, $\Max(D)^\inverse$ is also Hausdorff and has a basis of clopen subsets, as in this case the inverse topology coincides with the so-called \emph{constructible} (or \emph{patch}) topology \cite[Corollary 4.4.9]{spectralspaces-libro}. See \cite[\textsection 1.3 and 1.4]{spectralspaces-libro} for a deeper study of these two topologies.

\section{Smooth sequences}\label{sect:smooth}
\begin{defin}\label{defin:smooth}
Let $\mathcal{X}=\{X_\alpha\}_{\alpha<\lambda}$ be a well-ordered sequences of sets. We say that $\mathcal{X}$ is a \emph{smooth sequence} if
\begin{itemize}
\item $X_\alpha\subseteq X_\beta$ whenever $\alpha\leq\beta$;
\item $\displaystyle{X_\beta=\bigcup_{\alpha<\beta}X_\alpha}$ for every limit ordinal $\beta$.
\end{itemize}
\end{defin}
One can define an analogous notion with descending sequences, changing $\subseteq$ with $\supseteq$ in the first bullet point above and union with intersection in the second point; in this case, we say that we have a \emph{descending smooth sequence}. When dealing with both notions at the same time, we use the term ``\emph{ascending} smooth sequence'' for what we called above simply ``smooth sequence''.

\begin{ex}
~\begin{enumerate}
\item Let $\{X_n\}_{n\inN}$ be any ascending sequence of sets. Setting $X_\omega:=\bigcup_{n\inN}X_n$, we have that $\{X_n\}_{n\leq\omega}$ is a smooth sequence of sets.
\item Let $\{\alpha\}_{\alpha<\lambda}$ be an initial segment of ordinal numbers. Then, considering ordinal numbers set-theoretically, $\{\alpha\}_{\alpha<\lambda}$ is a smooth sequence.
\item The derived sequence of a topological space is a descending smooth sequence (see Section \ref{sect:ovop} below).
\item Let $\Theta$ be a pre-Jaffard family of an integral domain (see \cite{jaff-derived}). Then, the derived sequence $\{T_\alpha\}$ of $\Theta$ is a smooth sequence \cite[Lemma 7.1]{locpic}.
\end{enumerate}
\end{ex}

Let now $D$ be a Pr\"ufer domain. To any $T\in\Over(D)$ we can associate two subsets of $\Spec(D)$, namely
\begin{equation*}
X(T):=\{P\in\Spec(D)\mid PT=T\}
\end{equation*}
and
\begin{equation*}
Y(T):=\{P\in\Spec(D)\mid PT\neq T\}.
\end{equation*}
These two sets form a partition of $\Spec(D)$; with respect to the inverse topology, $X(T)$ is open and $Y(T)$ is closed. Conversely, if $Y$ is a closed set of $\Spec(D)^\inverse$, we can define an overring by
\begin{equation*}
\Theta(T):=\bigcap\{D_P\mid P\in Y\}.
\end{equation*}
Then, we have $Y(\Theta(Y))=Y$ and $\Theta(Y(T))=T$; thus, overrings of $D$ correspond to closed sets of $\Spec(D)^\inverse$ (see \cite[Proposition 5.2]{Xx}; note that the map $P\mapsto D_P$ establishes a bijection between $\Spec(D)$ and the Zariski space $\mathrm{Zar}(D)$ since $D$ is Pr\"ufer \cite[Theorem 26.1]{gilmer}). In the context of smooth sequences, these ideas give rise to the following criterion.
\begin{prop}\label{prop:smooth-caratt}
Let $D$ be a Pr\"ufer domain and let $\mathcal{D}:=\{D_\alpha\}_{\alpha<\lambda}$ be a well-ordered ascending chain of overrings of $D$. For every $\alpha$, let $X_\alpha:=X(D_\alpha)$ and $Y_\alpha:=Y(D_\alpha)$ (following the notation above). Then, the following are equivalent:
\begin{enumerate}[(i)]
\item\label{prop:smooth-caratt:D} $\mathcal{D}$ is an ascending smooth sequence of rings;
\item\label{prop:smooth-caratt:X} $\mathcal{X}:=\{X_\alpha\}$ is an ascending smooth sequence of sets;
\item\label{prop:smooth-caratt:Y} $\mathcal{Y}:=\{Y_\alpha\}$ is a descending smooth sequence of sets;
\end{enumerate}
\end{prop}
\begin{proof}
Without loss of generality we can suppose that  $D=D_0$; note also that each $D_\alpha$ is a Pr\"ufer domain, as an overring of a Pr\"ufer domain. As $X_\alpha=\Spec(D)\setminus Y_\alpha$, the equivalence of \ref{prop:smooth-caratt:X} and \ref{prop:smooth-caratt:Y} is obvious; thus, we only prove that \ref{prop:smooth-caratt:D} $\iff$ \ref{prop:smooth-caratt:X}.

\ref{prop:smooth-caratt:D} $\Longrightarrow$ \ref{prop:smooth-caratt:X} It is trivial to see that $\mathcal{X}$ is a descending sequence. Let $\beta$ be a limit ordinal and $P\in X_\beta$: then, $PD_\beta=D_\beta$ and thus $1=p_1x_1+\cdots+p_nx_n$ for some $p_i\in P$, $x_i\in D_\beta$. Since $\mathcal{D}$ is smooth, there is an $\alpha<\beta$ such that $x_1,\ldots,x_n\in D_\alpha$: then, $PD_\alpha=D_\alpha$ and $P\in X_\alpha$. Hence $\bigcup_{\alpha<\beta}X_\alpha=X_\beta$, as claimed.

\ref{prop:smooth-caratt:X} $\Longrightarrow$ \ref{prop:smooth-caratt:D} Let $\beta$ be a limit ordinal and let $d\in D_\beta$. If $d\in D$ then $d\in D_\alpha$ for every $\alpha<\beta$. If $d\notin D$, let $I:=(D:_Dd)=d^{-1}D\cap D$: then, $D$ is a finitely generated proper ideal of $D$, and thus $\V(I)$ is an open subset of $\Spec(D)^\inverse$. Since $d\in D_\beta$, we have $ID_\beta=D_\beta$; hence, no $P\in Y_\beta$ can belong to $\V(I)$, that is, $\V(I)\cap Y_\beta=\emptyset$. However, $\{\V(I)\cap Y_\alpha\}$ is a family of closed subsets of $\V(I)$ with the finite intersection property (since $\{Y_\alpha\}$ is a chain), and
\begin{equation*}
\bigcap_{\alpha<\beta} \V(I)\cap Y_\alpha=\V(I)\cap\bigcap_{\alpha<\beta} Y_\alpha=\V(I)\cap Y_\beta=\emptyset.
\end{equation*}
by smoothness. Since $\V(I)$ is closed in the Zariski topology, is also closed in the constructible topology; the latter is Hausdorff, and thus $\V(I)$ is compact in the constructible topology and also in the inverse topology, which is coarser than the constructible topology (see \cite[\textsection 1.3, 1.4]{spectralspaces-libro}). It follows that $\V(I)\cap Y_{\overline{\alpha}}=\emptyset$ for some $\overline{\alpha}<\beta$; hence, $ID_{\overline{\alpha}}=D_{\overline{\alpha}}$ and $d\in D_{\overline{\alpha}}$. Thus $D_\beta=\bigcup_{\alpha<\beta}D_\alpha$ and $\mathcal{D}$ is smooth.
\end{proof}

In particular, the smooth sequences starting from a Pr\"ufer domain $D$ and contained in $\quot(D)$ depend uniquely on $\Spec(D)$, in the following sense. Let $D'$ be another Pr\"ufer domain and suppose there is a homeomorphism $\phi:\Spec(D)\longrightarrow\Spec(D')$. Consider a smooth sequence $\mathcal{D}:=\{D_\alpha\}_{\alpha<\lambda}$ with $D_0=D$ and $D_\alpha\subseteq\quot(D)$ for every $\alpha$: by the previous proposition (and in the same notation), $\mathcal{D}$ correspond to a descending smooth sequence $\{Y_\alpha\}$ of subsets of $\Spec(D)$. If we define
\begin{equation*}
D'_\alpha:=\bigcap_{P\in\phi(Y_\alpha)}D'_P
\end{equation*}
we obtain a smooth sequence $\mathcal{D}':=\{D'_\alpha\}$ with $D'_0=D'$ and $D'_\alpha\subseteq\quot(D')$ for every $\alpha$. Moreover, $\phi$ induces homeomorphisms $\phi_\alpha:\Spec(D_\alpha)\longrightarrow\Spec(D'_\alpha)$ for every $\alpha$, since the prime ideals of $D$ that survive in $D_\alpha$ correspond exactly to the prime ideals of $D'$ that survive in $D'_\alpha$.

\section{Overring operators}\label{sect:ovop}

A classical example of a (descending) smooth sequence of sets is given by the derived sequence of a topological space: if $X$ is a topological space, let $\deriv(X)$ be the set of non-isolated points of $X$. Define $\deriv^0(X):=X$ and, for all ordinal numbers $\alpha\geq 1$, define:
\begin{equation*}
\deriv^\alpha(X):=\begin{cases}
\deriv(\deriv^\beta(X)) & \text{if~}\alpha=\beta+1\text{~is a successor ordinal},\\
\bigcap_{\beta<\alpha}\deriv^\beta(X) & \text{if~}\alpha\text{~is a limit ordinal}.
\end{cases}
\end{equation*}
The \emph{Cantor-Bendixson} rank of $X$ is the smallest ordinal number $\alpha$ such that $\deriv^\alpha(X)=\deriv^{\alpha+1}(X)$.

In this paper, we are mainly interested in smooth sequences of domains, and the easiest way to construct them is to mimic the previous construction by recursively iterating the same operator. We thus introduce the following definition.
\begin{defin}
Let $\class$ be a class of integral domains. We say that a map $\Lambda:\class\longrightarrow\class$, $D\mapsto\Lambda D$ is an \emph{overring operator} on $\class$ if $\Lambda D\in\Over(D)$ for every $D\in\class$.
\end{defin}
In particular, if $\Lambda$ is an overring operator then $D\subseteq\Lambda D$ for all $D\in\class$.

We say that a class $\class$ of integral domains is \emph{closed by overrings} if $\Over(D)\subseteq\class$ for every $D\in\class$. In this case, we can use an overring operator $\Lambda$ on $\class$ to construct, from any $D\in\class$, a smooth sequence $\{\Lambda^\alpha D\}$ in the following way. Fix $D\in\class$ and let $\alpha$ be an ordinal number. We define recursively:
\begin{itemize}
\item $\Lambda^0D:=D$;
\item if $\alpha=\beta+1$ is a successor ordinal, then
\begin{equation*}
\Lambda^\alpha D=\Lambda(\Lambda^\beta D);
\end{equation*}
\item if $\alpha$ is a limit ordinal, then
\begin{equation*}
\Lambda^\alpha D:=\bigcup_{\beta<\alpha}\Lambda^\beta D.
\end{equation*}
\end{itemize}
Note that the union in the last bullet point is a ring, since $\{\Lambda^\beta D\}_{\beta<\alpha}$ is a chain, and is in $\class$, since $\class$ is closed by overrings. It follows that $\{\Lambda^\alpha D\}_\alpha$ is a well-defined chain, and by construction it is a smooth sequence. We call it the \emph{smooth sequence of $D$ associated to $\Lambda$}.

\begin{defin}
Let $\class$ be a class of integral domains that is closed by overrings and let $\Lambda$ be an overring operator over $\class$. For every $D\in\class$, the \emph{$\Lambda$-rank of $D$} is
\begin{equation*}
\Lambda\mh\rk(D):=\min\{\alpha\mid \Lambda^\alpha D=\Lambda^{\alpha+1}D\}.
\end{equation*}
We say that $D$ is \emph{$\Lambda$-scattered} if $\Lambda^\alpha D=\quot(D)$ for some $\alpha$.

We say that an ordinal $\alpha$ is \emph{$\Lambda$-realizable} on $\class$ if there is a $D\in\class$ such that $\Lambda\mh\rk(D)=\alpha$.
\end{defin}
If $\Lambda^\alpha D=\Lambda^{\alpha+1}D$, then we also have $\Lambda^{\alpha+1}D=\Lambda^{\alpha+2}D=\Lambda^{\alpha+3}D=\cdots$ and, consequently, $\Lambda^\beta D=\Lambda^\alpha D$ for every $\beta>\alpha$. It follows that the cardinality of $\Lambda\mh\rk(D)$ is at most equal to the cardinality of $\Over(D)$ or, even more precisely, to the largest cardinality among the well-ordered chains in $\Over(D)$ starting from $D$. In particular, $\Lambda\mh\rk(D)$ is always well-defined.

\begin{defin}
Let $\class$ be a class of integral domains that is closed by overrings and let $\Lambda$ be an overring operator over $\class$.  The \emph{rank of $\Lambda$ over $\class$} is
\begin{equation*}
\rk_\class(\Lambda):=\sup\{\Lambda\mh\rk(D)\mid D\in\class\}
\end{equation*}
with $\rk_\class(\Lambda):=\infty$ if the supremum does not exists, i.e., if arbitrarily large ordinal numbers are realizable as $\Lambda$-ranks.
\end{defin}

\begin{oss}
Let $\class$ be a class of integral domains that is closed by overrings and let $\Lambda$ be an overring operator over $\class$. 
\begin{enumerate}
\item A domain $D$ has $\Lambda$-rank $0$ if and only if $D=\Lambda D$, i.e., if and only if $D$ is a fixed point of $\Lambda$. We call such a domain \emph{$\Lambda$-closed}.
\item On the opposite side, if $\Lambda D=\quot(D)$, we say that $D$ is \emph{$\Lambda$-trivial}.
\item If $\Lambda$ has rank $0$, then every $D\in\class$ is $\Lambda$-closed. Thus the identity is the unique overring operator of rank $0$.
\item An operator $\Lambda$ has rank $1$ if and only if $\Lambda D=\Lambda^2D$ for every $D\in\class$, i.e., if and only if $\Lambda$ is idempotent. If we restrict to operators that are order-preserving (i.e., such that $A\subseteq B$ implies $\Lambda A\subseteq\Lambda B$) then $\Lambda$ has rank $1$ if and only if it restricts to a closure operator on $\Over(D)$ for every $D\in\class$.
\end{enumerate}
\end{oss}

Let now $\Lambda$ be an overring operator on a class $\class$ of Pr\"ufer domains, and fix $D\in\class$. Since the overrings of $D$ correspond to the closed sets of $\Spec(D)^\inverse$ (see Section \ref{sect:smooth} above), we can interpret overring operators topologically.

Let $\mathcal{X}(\Spec(D)^\inverse)$ be the \emph{hyperspace} of $\Spec(D)^\inverse$, i.e., the set of closed subsets of $\Spec(D)^\inverse$. The action of $\Lambda$ on $\Over(D)$ can be be thought of as a map
\begin{equation*}
\Psi:\mathcal{X}(\Spec(D)^\inverse)\longrightarrow\mathcal{X}(\Spec(D)^\inverse):
\end{equation*}
indeed, to a closed set $Z$ we can associate the overring $\Theta(Z)$, and we can set
\begin{equation*}
\Psi(Z):=Y(T_Z),
\end{equation*}
using the notation introduced before Proposition \ref{prop:smooth-caratt}.

The smooth sequence $\{\Lambda^\alpha D\}$ associated to $\Lambda$ can also be recovered from $\Psi$ through a sequence $\{\Psi^\alpha\}$ of self-maps of $\mathcal{X}(\Spec(D)^\inverse)$ defined as follows:
\begin{itemize}
\item $\Psi^0$ is the identity map;
\item if $\alpha=\beta+1$ is a successor ordinal, $\Psi^\alpha:=\Psi\circ\Psi^\beta$;
\item if $\alpha$ is a limit ordinal,
\begin{equation*}
\Psi^\alpha(X)=\bigcap_{\beta<\alpha}\Psi^\beta(X)
\end{equation*}
for every $X\in\mathcal{X}(\Spec(D)^\inverse)$.
\end{itemize}
Under this notation, we have $\Lambda^\alpha D=\bigcap\{D_P\mid P\in\Psi^\alpha(\Spec(D))\}$.

\begin{ex}
~\begin{enumerate}
\item Let $\class$ be the class of all integral domains. Setting $\Lambda D$ to be equal to the integral closure of $D$ in its quotient field, we obtain that $\Lambda:\class\longrightarrow\class$ is an overring operator; the $\Lambda$-closed domains are exactly the integrally closed domains, while only the fields are $\Lambda$-trivial. Since the integral closure of a domain is integrally closed, $\Lambda$ has rank $1$ on $\class$ and is a closure.

\item Let $\class$ be the class of all integral domains, and let $\Lambda D$ be the complete integral closure of $D$. Then, $\Lambda$ is not idempotent; indeed, $\rk_\class(\Lambda)=\omega_1$, the first uncountable ordinal, and  every countable ordinal number is realizable as a $\Lambda$-rank \cite{cic-ordinal}. In fact, the results in \cite{cic-ordinal} show that the same happens if one restricts to the class $\class'$ of the Pr\"ufer domains of dimension at most $2$.

\item\label{ex:derived} Let $\class$ be the class of all Pr\"ufer domains of dimension at most $1$. For every $D\in\class$ that is not a field, define
\begin{equation*}
\Lambda D:=\bigcap_{P\in\deriv(\Max(D)^\inverse)}D_P,
\end{equation*}
where $\deriv(X)$ is the derived set of the topological space $X$ and we consider the empty intersection equal to $\quot(D)$. Then, $\Lambda$ is an overring operator on $\class$, and by the reasoning above we have
\begin{equation*}
\Lambda^\alpha D=\bigcap_{P\in\deriv^\alpha(\Max(D)^\inverse)}D_P
\end{equation*}
for every ordinal number $\alpha$. In particular, $\Lambda\mh\rk(D)$ is equal to the Cantor-Bendixson rank of $\Max(D)^\inverse$. Every ordinal number $\alpha$ is realizable, even in the smaller class of almost Dedekind domains \cite[Section 3]{olberding-factoring-SP}, and thus $\rk_\class(\Lambda)=\infty$.

\item Let $\class$ the class of all almost Dedekind domains. For every $D\in\class$, say that a maximal ideal $M$ of $D$ is \emph{critical} if it does not contain any nonzero radical finitely generated ideal, and denote by $\Crit(D)$ the set of critical ideals of $D$. Then, setting
\begin{equation*}
\Lambda D:=\bigcap_{M\in\Crit(D)}D_M
\end{equation*}
we obtain an overring operator $\Lambda$ on $\class$. The associated smooth sequence $\{\Lambda^\alpha D\}$ has been introduced as the \emph{SP-derived sequence} of $D$ \cite{SP-scattered}, and it always reaches $\quot(D)$ \cite[Theorem 5.1]{bounded-almded}, i.e., every almost Dedekind domain $D$ is $\Lambda$-scattered. Almost Dedekind domains of rank $1$ (thus, such that $\Lambda D=K$, or such that $\Crit(D)=\emptyset$) are called \emph{SP-domains}, and are exactly the integral domains where every ideal can be factored as a finite product of radical ideals \cite{vaughan-SP} (see also \cite[Section 3.1]{fontana_factoring}). Every ordinal number $\alpha$ is realizable as a $\Lambda$-rank \cite{almded-graph}, and thus $\rk_\class(\Lambda)=\infty$.

\item Let $\class$ be the class of all Pr\"ufer domains, and define
\begin{equation*}
\Lambda D:=\bigcap_{P\in\Spec(D)\setminus\Max(D)}D_P
\end{equation*}
for all $D\in\class$. Then, $\Lambda$ is an overring operator, and the prime ideals $P$ of $D$ that explode in $\Lambda D$ are exactly the elements of the interior of $\Max(D)^\inverse$. In particular, a Pr\"ufer domain $D$ is $\Lambda$-closed if and only if $\Spec(D)\setminus\Max(D)$ is dense in $\Spec(D)$, with respect to the inverse topology; this happens, for example, when $D$ is a valuation domain with unbranched maximal ideal. If $V$ is an $n$-dimensional valuation domain, then $V$ is scattered with $\Lambda$-rank $n$. It is not known if $\Lambda$ is scattered when restricted to the class $\class'$ of strongly discrete Pr\"ufer domains; if this is true, the group of invertible ideals of every $D\in\class'$ is a free group (see \cite[Theorem 8.6 and Proposition 9.1]{InvXD}).

\item For an integral domain $D$, let $\mathcal{P}_D$ be the set of all finite products of prime elements of $D$ (including the empty product, which we set equal to $1$). Then, $\mathcal{P}_D$ is a multiplicatively closed set, and the assignment
\begin{equation*}
\Lambda D:=\mathcal{P}_D^{-1}D
\end{equation*}
defines an overring operator on $\class$. A domain $D$ is $\Lambda$-closed if it has no prime elements, while it is $\Lambda$-trivial if and only if it is a UFD (since a domain is a UFD if and only if every nonzero prime ideal contains a prime element \cite[Theorem 5]{kaplansky}).

Consider now the same operator, but on the class $\class_1$ of all Dedekind domains. In this case, $\Lambda D$ has no prime elements \cite{dedekind-noprimes}; therefore, $\Lambda$ is idempotent on $\class_1$, and thus it has rank $1$.

On the other hand, consider $\Lambda$ on the class $\class_2$ of all B\'ezout domains of dimension $1$. Then, a maximal ideal $M$ contain a prime element if and only if it is principal, if and only if it is finitely generated, if and only if $M$ is an isolated point of $\Max(D)^\inverse$. It follows that on $\class_2$ the operator $\Lambda$ coincides with the derived set operator considered in point \ref{ex:derived}, and thus $\rk_{\class_2}(\Lambda)=\infty$. In particular, also $\rk_\class(\Lambda)=\infty$.

\item If $0<\alpha<\rk_{\class}(\Lambda)$, it is possible that $\alpha$ is not $\Lambda$-realizable. For example, let $V$ be an infinite-dimensional valuation domain such that the non-zero prime ideals form a descending chain $P_0\supset P_1\supset\cdots$ (where $P_0$ is the maximal ideal of $V$); then, the overrings of $V$ are $\quot(V)$ and $W_n:=V_{P_n}$ for $n\inN$. Let $\class:=\Over(V)$ and define an overring operator $\Lambda$ on $\class$ by setting $\Lambda W_n:=W_{n+1}$, $\Lambda\quot(V):=\quot(V)$. Then, $\quot(V)$ is the unique $\Lambda$-closed ring and each $W_n$ has $\Lambda$-rank $\omega$; thus, the unique $\Lambda$-realizable ordinal numbers are $0$ and $\omega$.
\end{enumerate}
\end{ex}

\section{Groups of invertible ideals}
In this section, we will use smooth sequences to investigate the freeness of subgroups of invertible ideals, in particular of the kernels of the extension maps. We shall use the following result (see for example \cite[Chapter 3, Lemma 7.4]{fuchs-abeliangroups}); recall that a subgroup $H$ of $G$ is \emph{pure} if $nG\cap H=nH$ for every $n\inN$.
\begin{prop}\label{prop:fuchs}
Let $G$ be an abelian group and let $\{H_\alpha\}_{\alpha<\lambda}$ be a smooth sequence of pure subgroups of $G$. If $\bigcup_\alpha H_\alpha=G$ and $L_\alpha:=H_{\alpha+1}/H_\alpha$ is free for every $\alpha<\lambda$, then $\displaystyle{G=\bigoplus_{\alpha<\lambda}L_\alpha}$. In particular, $G$ is free.
\end{prop}

\begin{lemma}\label{lemma:pure}
Let $D$ be an integral domain and $T\in\Over(D)$; let $\phi:\Inv(D)\longrightarrow\Inv(T)$ be the extension map. Then, $\ker\phi$ is pure in $\Inv(D)$.
\end{lemma}
\begin{proof}
Every $I\in\Inv(D)$ can be written as $JL^{-1}$ for some proper ideals $J,L\in\Inv(D)$ (for example, let $a\in D\setminus\{0\}$ such that $aI\subseteq D$ and take $J=aI$, $L=aD$). Hence, we can suppose without loss of generality that $I\subseteq D$.

Let $n\inN$; clearly $n\ker\phi\subseteq n\Inv(D)\cap \ker\phi$. If $I\in n\Inv(D)\cap \ker\phi$, then $IT=T$ and $I=J^n$ for some $J\in\Inv(D)$. If $JT\neq T$, then $J\subseteq M$ for some $M\in\Max(T)$, and thus also $I\subseteq M$ and $IT\neq T$, a contradiction. Thus $JT=T$, i.e., $J\in\ker\phi$ and $I\in n\ker\phi$.
\end{proof}

\begin{prop}\label{prop:smooth-caratt-inv}
Let $D$ be an integral domain and let $\mathcal{D}:=\{D_\alpha\}_{\alpha<\lambda}$ be a well-ordered ascending chain of overrings of $D$. For every $\alpha$, let $\pi_\alpha:\Inv(D_0)\longrightarrow\Inv(D_\alpha)$ be the natural extension map. Consider the following properties:
\begin{enumerate}[(i)]
\item\label{prop:smooth-caratt-inv:ring} $\mathcal{D}$ is a smooth sequence of rings;
\item\label{prop:smooth-caratt-inv:grp} $\mathcal{K}:=\{\ker(\pi_\alpha)\}_{\alpha<\lambda}$ is a smooth sequence of subgroups of $\Inv(D_0)$.
\end{enumerate}
Then, \ref{prop:smooth-caratt-inv:ring} $\Longrightarrow$ \ref{prop:smooth-caratt-inv:grp}. If $D$ is a Pr\"ufer domain, then \ref{prop:smooth-caratt-inv:ring} $\iff$ \ref{prop:smooth-caratt-inv:grp}.
\end{prop}

Note that the last sentence provides an equivalence that can be added to the ones of Proposition \ref{prop:smooth-caratt}.

\begin{proof}
Without loss of generality we can suppose that  $D=D_0$.

Suppose first that $\mathcal{D}$ is smooth; clearly $\mathcal{K}$ is ascending. Let $\beta$ be a limit ordinal and $A=(a_1,\ldots,a_n)\in\ker\pi_\beta$. Then, $A\subseteq D_\beta$ and $AD_\beta=D_\beta$. In particular, $1\in AD_\beta$, and thus there are $x_1,\ldots,x_n\in D_\beta$ such that $1=a_1x_1+\cdots+a_nx_n$. Since $\beta$ is a limit ordinal and $\mathcal{D}$ is smooth, there is an ordinal $\alpha<\beta$ such that $a_1,\ldots,a_n,x_1,\ldots,x_n\in D_\alpha$: then, $A\subseteq D_\alpha$ (as $a_i\in D_\alpha$) and $1\in AD_\alpha$ (as $x_i\in D_\alpha$), so that $AD_\alpha=D_\alpha$, i.e., $A\in\ker\pi_\alpha$. Thus $\mathcal{K}$ is smooth.

Suppose now that $D$ is a Pr\"ufer domain and that $\mathcal{K}$ is smooth. If $\mathcal{D}$ is not smooth, by Proposition \ref{prop:smooth-caratt} there is a limit ordinal $\beta$ and a prime ideal $P$ of $D$ such that $PD_\beta=D_\beta$ but $PD_\alpha\neq D_\alpha$ for every $\alpha<\beta$. Let $x_1,\ldots,x_n\in P$, $y_1,\ldots y_n\in D_\beta$ be such that $1=x_1y_1+\cdots+x_ny_n$: then, $I=(x_1,\ldots,x_n)$ is an invertible proper ideal of $D$ that belongs to $\ker\pi_\beta=\bigcup_{\alpha<\beta}\pi_\alpha$. Thus $ID_\alpha=D_\alpha$ for some $\alpha<\beta$, a contradiction. Hence $\mathcal{D}$ must be smooth, as claimed.
\end{proof}

\begin{oss}
If $D$ is not a Pr\"ufer domain, the two conditions of Proposition \ref{prop:smooth-caratt-inv} are not in general equivalent. Suppose that $F$ is a field and let $Y,Z,X_1,\ldots,X_n,\ldots$ be independent indeterminates over $F$. Let $L:=K(Y,X_1,\ldots,X_n,\ldots)$ and, for every $n$, let $D_n:=K[X_1,\ldots,X_n]+ZL[[Z]]$. Then, each $D_n$ is an overring of $D_0=K+ZL[[Z]]$; moreover, since $K[X_1,\ldots,X_n]$ is a UFD, every invertible ideal of $D_n$ is principal. In particular, the kernel $G_{n,m}$ of every extension map $\Inv(D_n)\longrightarrow\Inv(D_m)$ (with $n<m$) is trivial, since $D_n$ and $D_m$ have the same units.

Let $D_\omega:=K[Y,X_1,\ldots,X_n,\ldots,]$. Similarly, each $G_{n,\omega}:\ker(\Inv(D_n)\longrightarrow\Inv(D_\omega))$ is trivial, and thus $\{G_{0,n}\}_{n<\omega+1}$ is a smooth sequence of subgroups of $\Inv(D_0)$. However, $\{D_n\}_{n<\omega+1}$ is not a smooth sequence of rings, since $\bigcup_{n<\omega}D_n=K[X_1,\ldots,X_n,\ldots,]+ZL[[Z]]$ does not contain $Y$.
\end{oss}

For the remainder of this section, when $\mathcal{D}:=\{D_\alpha\}_{\alpha<\lambda}$ is a smooth sequence of overrings and $\alpha\leq\beta<\lambda$, we set
\begin{equation*}
\pi_{\alpha,\beta}:\Inv(D_\alpha)\longrightarrow\Inv(D_\beta)
\end{equation*}
to be the natural extension map.

We now want to transform Proposition \ref{prop:smooth-caratt-inv} into a sufficient criterion for freeness. We start by studying surjectivity.
\begin{prop}\label{prop:smooth-surjective}
Let $\mathcal{D}:=\{D_\alpha\}_{\alpha<\lambda}$ be a smooth sequence of overrings of $D:=D_0$, and let $\beta<\lambda$. If $\pi_{\alpha,\alpha+1}$ is surjective for every $\alpha<\beta$, then $\pi_{0,\beta}$ is surjective.
\end{prop}
\begin{proof}
We proceed by induction on $\beta$. If $\beta=0$ the claim is trivial. If $\beta=\alpha+1$ is a successor ordinal, then $\pi_{0,\beta}=\pi_{\alpha,\beta}\circ\pi_{0,\alpha}$ is the composition of two surjective maps and thus it is surjective.

Suppose that $\beta$ is a limit ordinal, and let $A=(a_1,\ldots,a_n)D_\beta\in\Inv(D_\beta)$. Then, there is a $B=(b_1,\ldots,b_m)D_\beta\in\Inv(D_\beta)$ such that $AB=D_\beta$; therefore, $a_ib_j\in D_\beta$ for every $i,j$ and there are $x_{ij}\in D_\beta$ such that
\begin{equation*}
1=\sum_{\substack{1\leq i\leq n\\ 1\leq j\leq m}}x_{ij}a_ib_j.
\end{equation*}
Since $\beta$ is a limit ordinal and $\mathcal{D}$ is smooth, there is an $\alpha<\beta$ such that $a_ib_j,x_{ij}\in D_\alpha$ for every $i,j$; in particular, $A':=(a_1,\ldots,a_n)D_\alpha$ and $B':=(b_1,\ldots,b_m)D_\alpha$ are fractional ideals of $D_\alpha$ such that $A'B'\subseteq D_\alpha$ and $1\in A'B'$, i.e., $A'B'=D_\alpha$. In particular, $A'\in\Inv(D_\alpha)$ satisfies $\pi_{\alpha,\beta}(A')=A$. By induction, $A'$ is in the image of $\pi_{0,\alpha}$, and thus $A$ is in the image of $\pi_{0,\beta}=\pi_{\alpha,\beta}\circ\pi_{0,\alpha}$. Thus $\pi_{0,\beta}$ is surjective, as claimed.
\end{proof}

\begin{teor}\label{teor:rings-smooth}
Let $D$ be an integral domain, and let $\{D_\alpha\}_{\alpha<\lambda}$ be a smooth sequence of overrings of $D$ such that $D=D_0$. Suppose that, for every $\alpha<\beta$, the map $\pi_{\alpha,\alpha+1}$ is surjective with free kernel. Then, we have an exact sequence
\begin{equation*}
0\longrightarrow\bigoplus_{\alpha<\beta}\ker\pi_{\alpha,\alpha+1}\longrightarrow\Inv(D)\longrightarrow\Inv(D_\beta)\longrightarrow 0.
\end{equation*}
In particular:
\begin{enumerate}[(a)]
\item $\ker\pi_{0,\beta}$ is free;
\item if $\Inv(D_\beta)$ is free, then $\Inv(D)$ is free.
\end{enumerate}
\end{teor}
\begin{proof}
By Proposition \ref{prop:smooth-surjective}, the extension map $\Inv(D)\longrightarrow\Inv(D_\beta)$ is surjective.

We prove the decomposition of $\ker\pi_{0,\beta}$ by induction on $\beta$. If $\beta=0$ the claim is trivial; if $\beta=1$ the freeness of $\ker\pi_{0,1}$ is just an hypothesis. Suppose that the claim holds for every ordinal number strictly smaller than $\beta$. If $\beta=\alpha+1$ is a successor ordinal, we have a chain of group homomorphisms
\begin{equation*}
\Inv(D)\xrightarrow{~~\pi_{0,\alpha}~~}\Inv(D_\alpha)\xrightarrow{~~\pi_{\alpha,\beta}~~}\Inv(D_{\beta+1})
\end{equation*}
whose composition is $\pi_{0,\beta}$; moreover, they are all surjective ($\pi_{0,\alpha}$ and $\pi_{0,\beta}$ by Proposition \ref{prop:smooth-surjective}, $\pi_{\alpha,\beta}$ by hypothesis). As in \cite[Lemma 5.7]{SP-scattered}, we have an exact sequence
\begin{equation*}
0\longrightarrow\ker\pi_{0,\alpha}\longrightarrow\ker\pi_{0,\beta}\longrightarrow\ker\pi_{\alpha,\beta}\longrightarrow 0.
\end{equation*}
Since $\ker\pi_{\alpha,\beta}$ is free, the sequence splits, and an application of the induction hypothesis shows that $\ker\pi_{0,\beta}$ has the claimed decomposition.

Suppose that $\beta$ is a limit ordinal. Using the previous reasoning, for every $\alpha<\beta$ the quotient between $\ker\pi_{0,\alpha+1}$ and $\ker\pi_{0,\alpha}$ is the free group $\ker\pi_{\alpha,\alpha+1}$. The claim now follows applying Proposition \ref{prop:fuchs} to the sequence $\{\ker\pi_{0,\alpha}\}_{\alpha<\beta}$, which is smooth by Proposition \ref{prop:smooth-caratt-inv} (note that each $\ker\pi_{0,\alpha}$ is pure by Lemma \ref{lemma:pure}).

The two ``in particular'' statements are now straightforward.
\end{proof}

\begin{cor}\label{cor:successione}
Let $D$ be an integral domain and $\mathcal{D}:=\{D_\alpha\}_{\alpha\leq\lambda}$ be a smooth sequence starting from $D_0:=D$. If $\pi_{\alpha,\alpha+1}$ is surjective with free kernel for every $\alpha<\lambda$, then $\Inv(D)\longrightarrow\Inv(D_\lambda)$ is surjective with free kernel.
\end{cor}
\begin{proof}
Immediate from Theorem \ref{teor:rings-smooth}.
\end{proof}

\begin{cor}\label{cor:rings-smooth}
Let $D$ be an integral domain. If there is a smooth sequence $\{D_\alpha\}_{\alpha\leq\lambda}$ such that:
\begin{itemize}
\item $D=D_0$ and $\quot(D)=D_\lambda$;
\item $\pi_{\alpha,\alpha+1}:\Inv(D_\alpha)\longrightarrow\Inv(D_{\alpha+1})$ is surjective and has free kernel for every $\alpha<\lambda$;
\end{itemize}
then $\Inv(D)$ is free.
\end{cor}
\begin{proof}
Under the hypothesis, $\Inv(D_\lambda)=(0)$. Apply Theorem \ref{teor:rings-smooth}.
\end{proof}

Theorem \ref{teor:rings-smooth} abstracts the proof of several similar results: \cite[Proposition 5.2]{inv-free} (for pre-Jaffard families of rings), \cite[Proposition 5.8]{SP-scattered} (for the sequence associated to critical ideals in almost Dedekind domains) and \cite[Theorems 6.8 and 8.6]{InvXD} (for strongly discrete Pr\"ufer domains); the proofs of these cases work (in the terminology of the present paper) by constructing an overring operator $\Lambda$ such that the kernel of $\Inv(D)\longrightarrow\Inv(\Lambda D)$ is free, and then use the construction in Section \ref{sect:ovop} to associate a smooth sequence to $\Lambda$. In particular, Corollary \ref{cor:rings-smooth} represent the better scenario, where $\Inv(D)$ is approximated well enough by the $\ker\pi_{0,\alpha}$ that we obtain that $\Inv(D)$ itself is free: this case happens, for example, for almost Dedekind domains (see \cite[Theorem 5.1]{bounded-almded}).

Beside its generality, the usefulness of Theorem \ref{teor:rings-smooth} lies in the fact that it can be applied to smooth sequences whose definition is peculiar to the domain under consideration. In the final two theorems of this paper, we show sufficient conditions under which some prime ideals can be ``thrown out'' when studying the freeness of $\Inv(D)$, where $D$ is a one-dimensional Pr\"ufer domain. Recall that $\Theta(Y)$ has been defined before Proposition \ref{prop:smooth-caratt}.
\begin{lemma}\label{lemma:Theta-decomp}
Let $D$ be a one-dimensional Pr\"ufer domain and let $A,B$ be nonempty closed subsets of $\Max(D)^\inverse$.
\begin{enumerate}[(a)]
\item\label{lemma:Theta-decomp:ext} The maximal ideals of $\Theta(A)$ are the extensions of the elements of $A$.
\item\label{lemma:Theta-decomp:free} If $\Inv(\Theta(A))$ and $\Inv(\Theta(B))$ are free, then $\Inv(\Theta(A\cup B))$ is free.
\item\label{lemma:Theta-decomp:jaffard} If $B=\Max(D)\setminus A$, then $\Inv(D)\simeq\Inv(\Theta(A))\oplus\Inv(\Theta(B))$
\end{enumerate}
\end{lemma}
\begin{proof}
\ref{lemma:Theta-decomp:ext} Let $A\subseteq\Max(D)^\inverse$ be closed: then, $A':=A\cup\{0\}$ is closed in $\Spec(D)^\inverse$. By the correspondence between closed sets of $\Spec(D)^\inverse$ and overrings of $D$, it follows that the elements of the spectrum of $\Theta(A)=\Theta(A')$ are exactly the extensions of the elements of $A'$. In particular, the maximal ideals of $\Theta(A')$ are the extensions of the maximal elements of $A'$, i.e., of the elements of $A$.

\ref{lemma:Theta-decomp:free} By the previous point, we can suppose without loss of generality that $A\cup B=\Max(D)$. Consider the map
\begin{equation*}
\begin{aligned}
\Phi\colon\Inv(D)  &\longrightarrow \Inv(\Theta(A))\oplus\Inv(\Theta(B)),\\
I & \longmapsto (I\Theta(A),I\Theta(B)).
\end{aligned}
\end{equation*}
Then, $\Phi$ is injective: indeed, suppose $I\in\ker\Phi$. Since $I\cap D\in\ker\Phi$ too, we can suppose without loss of generality that $I\subseteq D$. If $I\subsetneq D$, there is a maximal ideal $M$ containing $I$; since $\Max(D)=A\cup B$, moreover, $M$ must belong to $A$ or $B$. In the former case, $I\Theta(A)\subsetneq \Theta(A)$, in the latter $I\Theta(B)\subsetneq\Theta(B)$; in both cases $I\notin\ker\Phi$, a contradiction. Therefore, $\Phi(\Inv(D))$ is isomorphic to a subgroup of $\Inv(\Theta(A))\oplus\Inv(\Theta(B))$, which is free as both summands are free. Thus $\Inv(D)$ is free too.

\ref{lemma:Theta-decomp:jaffard} If $B=\Max(D)\setminus A$ (i.e., if $A$ and $B$ are disjoint) then a maximal ideal cannot survive in both $\Theta(A)$ and $\Theta(B)$; therefore, $\{\Theta(A),\Theta(B)\}$ is a Jaffard family of $D$ (see \cite[Section 6.3]{fontana_factoring} and \cite{starloc}) and thus the map $\Phi$ above is surjective, so that the group $\Inv(D)$ decomposes as claimed \cite[Proposition 7.1]{starloc}.
\end{proof}

\begin{teor}\label{teor:kernel-free-dvr}
Let $D$ be a one-dimensional Pr\"ufer domain and $T$ be an overring of $D$. If $D_P$ is a DVR for every $P\in\Spec(D)$ such that $PT=T$, then the kernel of the extension homomorphism $\Inv(D)\longrightarrow\Inv(T)$ is free.
\end{teor}
\begin{proof}
Let $Z:=\{P\in\Spec(T)\mid PT=T\}$, and choose a well-ordering of $Z$, say $Z=\{P_\alpha\}_{\alpha<\lambda}$. Since $Z$ is an open set of $\Max(D)^\inverse$, for every $\alpha$ we can find a clopen subset $Z_\alpha\subseteq Z$ such that $P_\alpha\in Z_\alpha$. In particular, $Z=\bigcup_\alpha Z_\alpha$.

For every ordinal number $\alpha<\lambda$, let $\displaystyle{Y_\alpha:=\Max(D)\setminus\bigcup_{\gamma<\alpha}Z_\gamma}$. Then, $\{Y_\alpha\}_{\alpha\leq\lambda}$ is a descending smooth sequence of closed subsets of $\Spec(D)^\inverse$ with $Y_\lambda=Y$, and thus by Proposition \ref{prop:smooth-caratt} the sequence $\mathcal{D}:=\{D_\alpha\}_{\alpha<\lambda}$ defined by
\begin{equation*}
D_\alpha:=\bigcap\{D_P\mid P\in Y_\alpha\}
\end{equation*}
is a smooth ascending sequence of overrings of $D$. Moreover, since $\bigcap_\alpha Y_\alpha=Y$, the union of all the $D_\alpha$ is $T$. 

Fix now any $\alpha<\lambda$. Then, we have
\begin{equation*}
H:=\{P\in\Max(D_\alpha)\mid PD_{\alpha+1}\neq D_{\alpha+1}\}=\{PD_\alpha\mid P\in Y_\alpha\setminus Z_{\alpha+1}\};
\end{equation*}
in particular, this set is a clopen subset of $\Max(D_\alpha)^\inverse$ and $D_{\alpha+1}=\Theta(H)$. By Lemma \ref{lemma:Theta-decomp}\ref{lemma:Theta-decomp:jaffard}, $\Inv(D_\alpha)\simeq\Inv(D_\alpha)\oplus\Inv(\Theta(H'))$, where $H':=\Max(D_\alpha)\setminus H$, and the kernel of the extension morphism $\Inv(D_\alpha)\longrightarrow\Inv(D_{\alpha+1})$ is isomorphic to $\Inv(\Theta(H'))$.

If $P$ is a maximal ideal of $\Inv(\Theta(H'))$, then $P$ is the extension of a maximal ideal of $Z_\alpha$, and $\Theta(H')_P=D_{P\cap D}$ is a DVR; thus $\Theta(H')$ is an almost Dedekind domain. By \cite[Proposition 5.3]{bounded-almded}, $\Inv(\Theta(H'))$ is free.

The claim now follows from Theorem \ref{teor:rings-smooth}.
\end{proof}

If $T=\quot(D)$, Theorem \ref{teor:kernel-free-dvr} reduces to the statement that $\Inv(D)$ is free when $D$ is an almost Dedekind domain, and thus the theorem is a generalization of \cite[Proposition 5.3]{bounded-almded}. However, the proof of Theorem \ref{teor:kernel-free-dvr} cannot substitute the proof of the almost Dedekind domain case, as the latter is used in the above proof.

\begin{teor}\label{teor:unione-numerabile}
Let $D$ be a one-dimensional Pr\"ufer domain, and let $\{A_n\}_{n\inN}$ be a countable set of clopen subsets of $\Max(D)^\inverse$; let $A:=\bigcup_nA_n$. If $\Inv(\Theta(A_n))$ is free for every $n$, then the kernel of $\Inv(D)\longrightarrow\Inv(\Theta(\Max(D)\setminus A))$ is free.
\end{teor}
\begin{proof}
For every $n\inN$, let $B_n:=\bigcup_{k=0}^nA_k$. For every $n$, $B_n$ is a clopen subset of $\Max(D)^\inverse$. Let $T_n:=\Theta(\Max(D)\setminus B_n)$; then, $\{T_n\}_{n\inN}$ is an ascending chain of overrings of $D$. By Lemma \ref{lemma:Theta-decomp}\ref{lemma:Theta-decomp:jaffard}, the kernel of $\Inv(D)\longrightarrow\Inv(T_n)$ is isomorphic to $\Inv(\Theta(B_n))$; applying Lemma \ref{lemma:Theta-decomp}\ref{lemma:Theta-decomp:free} and induction, we see that each $\Inv(\Theta(B_n))$ is free. 

Let $T:=\Theta(\Max(D)\setminus A)$; we claim that $\bigcup_{n\inN}T_n=T$. Indeed, by Proposition \ref{prop:smooth-caratt} it is enough to show that, for any $P\in\Max(D)$, we have
\begin{equation*}
PT=T\iff PT_n=T_n\text{~for some~}n\inN.
\end{equation*}
However, $PT_n=T_n$ if and only if $P\in B_n$, and thus the second condition above is equivalent to
\begin{equation*}
P\in\bigcup_{n\inN}B_n=\bigcup_{n\inN}A_n=A.
\end{equation*}
By definition of $T$, we have $PT=T$, as claimed.

Therefore, setting $T_\omega:=T$ the sequence $\mathcal{T}:=\{T_n\}_{n\leq\omega}$ is smooth. The claim now follows applying Corollary \ref{cor:successione} to $\mathcal{T}$.
\end{proof}

\end{document}